\documentclass[12pt,a4paper]{article}
\usepackage[]{amsmath,amssymb,amsthm}
\usepackage{epstopdf}
\usepackage{tikz}
\usepackage{pgfplots}
\pgfplotsset{compat=newest}
\usepgfplotslibrary{fillbetween}
\usepackage{subfigure}
\usepackage{hyperref}
\usetikzlibrary{arrows.meta}
\usepackage[section]{placeins}

\newtheorem{theorem}{Theorem}[section]
\newtheorem{proposition}[theorem]{Proposition}
\newtheorem{corollary}[theorem]{Corollary}

\theoremstyle{definition}
\newtheorem{definition}[theorem]{Definition}
\theoremstyle{remark}

\newcommand{\R}{\mathbb{R}}
\newcommand{\N}{\mathbb{N}}

\newcommand{\loc}{\textnormal{loc}}
\newcommand{\dx}{\textnormal{d}x}
\newcommand{\B}{{\cal B}}
\begin{document}
\begin{center}
{\large{\bf Stability Indices of Non-Hyperbolic Equilibria in Two-Dimensional Systems of ODEs}}\\
\mbox{} \\
\begin{tabular}{c}
{\bf Alexander Lohse} \\
{\small alexander.lohse@uni-hamburg.de}\\
{\small Fachbereich Mathematik, Universit\"at Hamburg}\\
{\small Bundesstra{\ss}e 55, 20146 Hamburg, Germany}\\
\end{tabular}
\end{center}

\begin{abstract}
We consider families of systems of two-dimensional ordinary differential equations with the origin $0$ as a non-hyperbolic equilibrium. For any number $s \in (-\infty, +\infty)$ we show that it is possible to choose a parameter in these equations such that the stability index $\sigma(0)$ is precisely $\sigma(0)=s$. In contrast to that, for a hyperbolic equilibrium $x$ it is known that either $\sigma(x)=-\infty$ or $\sigma(x)=+\infty$. Furthermore, we discuss a system with an equilibrium that is locally unstable but globally attracting, highlighting some subtle differences between the local and non-local stability indices.
\end{abstract}


\vspace{.3cm}



\section{Introduction}
Attraction and stability of invariant sets are crucial concepts in the qualitative theory of dynamical systems: the degree to which a set possesses these properties is directly linked to the way it influences the overall (longterm) dynamics of a system. Beyond the classic notion of asymptotic (Lyapunov) stability several levels of so-called \emph{non-asymptotic stability} have been identified. These include \emph{fragmentary asymptotic stability (f.a.s.)} \cite{Podvigina2012} and \emph{essential asymptotic stability (e.a.s.)} \cite{Melbourne1991} to mention probably the two most frequent ones. Loosely speaking, an f.a.s.\ set attracts something of positive measure while an e.a.s.\ set attracts ``almost everything" in a small neighbourhood.

In 2011 Podvigina and Ashwin \cite{PodviginaAshwin2011} introduced a \emph{(local) stability index} as a means of quantifying stability and attraction of invariant sets in discrete and continuous dynamical systems. It is linked to the stability properties mentioned above: roughly speaking, positive indices correspond to essential asymptotic stability, while fragmentary asymptotic stability is associated with indices that are greater than $-\infty$, see \cite{Lohse2015b} for a detailed discussion of this. In the last decade, this concept has been used to characterize various types of attractors, e.g.\ heteroclinic cycles/networks \cite{BickLohse2019, CastroLohse2014, GarridoDaSilvaCastro2019}, invariant graphs in skew product systems \cite{Keller2014} or attractors with riddled basins \cite{RoslanAshwin2016}.

For the simple case of a hyperbolic equilibrium the stability index does not reveal significant information, since it turns out to be either $+\infty$ (for a sink) or $-\infty$ (for a saddle or source). In this paper we discuss two families of ordinary differential equations on $\R^2$ that possess the origin $0$ as a non-hyperbolic equilibrium. We show that
\begin{itemize}
\item[(i)] for any given real number $s>0$ we can choose a parameter in the first family such that we obtain $\sigma(0)=s$, and
\item[(ii)] the same is possible for any $s<0$ in the second family.
\end{itemize}
This confirms that non-hyperbolic equilibria can indeed be f.a.s.\ or e.a.s\ without being asymptotically stable.

We also present an example of a smooth system with a non-hyperbolic equilibrium that is strongly attracting (stability index equal to $+\infty$) but at the same time locally repels most initial conditions (local stability index equal to $-\infty$). Systems with similar properties in previous work \cite{Lohse2015} lacked smoothness.

In higher-dimensional systems our results may be useful for understanding the dynamics along the centre manifold of an equilibrium, thus helping to better describe stability and attraction properties of non-hyperbolic steady states. Moreover, the way we design these systems might serve as a prototype for controlling stability indices in more involved settings, e.g.\ along heteroclinic connections.

The paper is organized as follows: in section~\ref{sec-prelim} we briefly discuss non-asymptotic stability and the (local) stability index. In section~\ref{sec-example} we present our examples and prove that the equilibria possess the desired stability indices. We conclude with some comments in section~\ref{sec-comments}.

\section{Preliminaries}\label{sec-prelim}
In this section we reproduce the definitions of fragmentary and essential asymptotic stability of a compact, invariant set $X \subset \R^n$ for a dynamical system on $\R^n$ given by $\dot x = f(x)$. Moreover, we recall the stability index that was introduced to quantify stability and attraction of such a set.

In line with standard notation we write $B_\varepsilon(x)$ for an $\varepsilon$-neighbourhood of a point $x \in \R^n$ and use $\ell(.)$ for Lebesgue measure. The basin of attraction of $X$, i.e.\ the set of points in $\R^n$ with $\omega$-limit set in $X$, is denoted by $\B(X)$. For $\delta >0$ the $\delta$-local basin of attraction $\B_\delta(X)$ is the subset of points in $\B(X)$ for which the trajectory never leaves $B_\delta(X)$ in positive time.

With this terminology we revisit the following definitions.

\begin{definition}[\cite{Podvigina2012}, definition 2]
$X$ is called \emph{fragmentarily asymptotically stable (f.a.s.)} if $\ell(\B_\delta(X))>0$ for any $\delta>0$.
\end{definition}

As discussed in \cite{Lohse2014} being f.a.s.\ is equivalent to having a basin of attraction of positive measure.

\begin{definition}[\cite{Brannath1994}, definition 1.2]
$X$ is called \emph{essentially asymptotically stable (e.a.s.)} if it is asymptotically stable relative to a set $N \subset \R^n$ which satisfies $$\lim\limits_{\varepsilon \to 0} \frac{\ell(B_\varepsilon(X) \cap N)}{\ell(B_\varepsilon(X))}=1.$$
\end{definition}

Here \emph{asymptotic stability relative to $N$} means that the usual conditions for asymptotic stability must be fulfilled for the intersection of a neighbourhood of $X$ with $N$, but not necessarily in an entire neighbourhood.

Note that in \cite{Melbourne1991} e.a.s.\ is used in the same sense as above, even though a slightly different definition is given.

\begin{definition}[\cite{PodviginaAshwin2011}, definition 5]
For $x \in X$ and $\varepsilon, \delta >0$ set 
\begin{equation*}
\Sigma_\varepsilon(x):=\frac{\ell(B_\varepsilon(x) \cap \B(X))}{\ell(B_\varepsilon(x))}, \qquad \Sigma_{\varepsilon,\delta}(x):=\frac{\ell(B_\varepsilon(x) \cap \B_\delta(X))}{\ell(B_\varepsilon(x))}.
\end{equation*}
Then the {\em stability index} at $x$ with respect to $X$ is defined as
\begin{equation*}
\sigma(x):=\sigma_+(x)-\sigma_-(x),
\end{equation*}
with
\begin{equation*}
\sigma_-(x):= \lim\limits_{\varepsilon \to 0}  \frac{\textnormal{ln}(\Sigma_\varepsilon(x) )}{\textnormal{ln}(\varepsilon)}, \qquad \sigma_+(x):= \lim\limits_{\varepsilon \to 0} \frac{\textnormal{ln}(1-\Sigma_\varepsilon(x) )}{\textnormal{ln}(\varepsilon)}.
\end{equation*}
The convention that $\sigma_-(x)=\infty$ if $\Sigma_\varepsilon(x)=0$ for some $\varepsilon>0$, and $\sigma_+(x)=\infty$ if $\Sigma_\varepsilon(x)=1$ for some $\varepsilon>0$, implies $\sigma(x) \in [-\infty, \infty]$.

Analogously, the {\em local stability index} at $x \in X$ is defined to be
\begin{equation*}
\sigma_\loc(x):=\sigma_{\textnormal{loc},+}(x)-\sigma_{\textnormal{loc},-}(x),
\end{equation*}
with
\begin{equation*}
\sigma_{\textnormal{loc},-}(x):= \lim\limits_{\delta \to 0} \lim\limits_{\varepsilon \to 0} \frac{\textnormal{ln}(\Sigma_{\varepsilon,\delta}(x))}{\textnormal{ln}(\varepsilon)}, \: \sigma_{\textnormal{loc},+}(x):= \lim\limits_{\delta \to 0} \lim\limits_{\varepsilon \to 0} \frac{\textnormal{ln}(1-\Sigma_{\varepsilon,\delta}(x))}{\textnormal{ln}(\varepsilon)}.
\end{equation*}
\end{definition}

For an invariant set $X \subset \R^n$ and a point $x \in X$ the index $\sigma(x)$ quantifies attraction to $X$ near $x$ in the system. In the same way the local index $\sigma_{\loc}(x)$ characterizes (Lyapunov) stability of $X$ near $x$. While these two properties often go hand in hand (and the local and non-local indices may coincide), it is well-known that they are independent of each other (so local and non-local indices may differ), see examples in \cite{Lohse2015}.

For a geometric intuition consider Figure~\ref{stabindex}: if $\sigma(x)>0$, then in a small neighbourhood of $x$ an increasingly large portion of points is contained in the basin of attraction $\B(X)$ and therefore attracted to $X$. If on the other hand $\sigma(x)<0$, then the portion of such points goes to zero as the neighbourhood $B_\varepsilon(x)$ shrinks. The meaning of signs for the local stability index may be illustrated analogously.

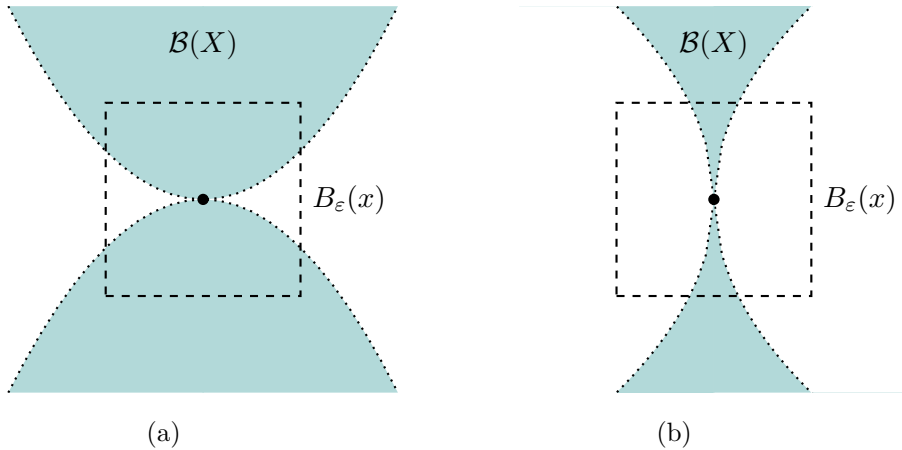
\begin{figure}
\centering
    \subfigure[]{
    \begin{tikzpicture}
    \begin{axis}[axis line style={draw=none}, tick style={draw=none},
    axis lines = middle, axis equal, scale=0.9,
    yticklabels={,,},
    xticklabels={,,}
    xmin=-1, xmax=1,
    ymin=-1, ymax=1]
\addplot [name path = A, thick, dotted, domain = -1:1] {x^2};
\addplot [name path = C, thick, dotted, domain = -1:1] {-x^2};
    \draw[teal!30, name path = V1] (0,0) -- (0,1);
    \draw[teal!30, name path = V2] (0,-1) -- (0,0);
    \path (-1,0) -- (1,0);
    \path[name path = H1] (-1,1) -- (0,1);
    \path[name path = H2] (0,-1) -- (1,-1);
    \addplot [teal!30] fill between [of = A and H1];
    \addplot [teal!30] fill between [of = C and H2];
\node at (0,0) [circle,fill,inner sep=1.5pt]{};
\node at (0,0.8) {\small $\B(X)$};
\node at (0.75,0) {\small $B_\varepsilon(x)$};
\draw[dashed, thick] (-0.5,-0.5) rectangle (0.5,0.5);
\end{axis}
\end{tikzpicture}
    }
    \subfigure[]{
    \begin{tikzpicture}
    \begin{axis}[axis line style={draw=none}, tick style={draw=none},
    axis lines = middle, axis equal, scale=0.9,
    yticklabels={,,},
    xticklabels={,,}
    xmin=-1, xmax=1,
    ymin=-1, ymax=1]
\addplot [name path = A, thick, dotted, domain = 0:1] {2*sqrt(x/2)};
\addplot [name path = B, thick, dotted, domain = -1:0] {2*sqrt(-x/2)};
\addplot [name path = C, thick, dotted, domain = 0:1] {-2*sqrt(x/2)};
\addplot [name path = D, thick, dotted, domain = -1:0] {-2*sqrt(-x/2)};
    \draw[teal!30, name path = V1] (0,0) -- (0,1);
    \draw[teal!30, name path = V2] (0,-1) -- (0,0);
    \path (-1,0) -- (1,0);
    \path[name path = H1] (-1,1) -- (0,1);
    \path[name path = H2] (0,-1) -- (1,-1);
    \addplot [teal!30] fill between [of = A and V1];
    \addplot [teal!30] fill between [of = B and H1];
    \addplot [teal!30] fill between [of = C and H2];
    \addplot [teal!30] fill between [of = D and V2];
\node at (0,0) [circle,fill,inner sep=1.5pt]{};
\node at (0,0.8) {\small $\B(X)$};
\node at (0.75,0) {\small $B_\varepsilon(x)$};
\draw[dashed, thick] (-0.5,-0.5) rectangle (0.5,0.5);
\end{axis}
\end{tikzpicture}
    }
\caption{(a) an e.a.s.\ equilibrium with a positive stability index; (b) an f.a.s.\ equilibrium with a negative stability index.}
\label{stabindex}
\end{figure}

Since here we are interested in the stability of equilibria, we typically have $X=\{0\}$ in the following, which prompts us to conveniently shorten our notation to $\B(0)=\B(\{0\})$ etc.

\section{Stability Indices}\label{sec-example}
In this section we discuss several families of systems in $\R^2$, each with a non-hyperbolic equilibrium which, depending on a parameter in the equations, may possess any given real number as its stability index. Note that we define the systems only for $x,y \geq 0$, but they can easily be symmetrically extended to the whole plane. Most of the time local and non-local stability indices coincide -- we therefore only distinguish between the two when this is not the case.

\subsection{Positive Stability Indices}
We first present a class of systems in $\R^2$ with the origin $0$ as an equilibrium that can have any stability index in $(0, +\infty)$. With a parameter $a>1$, for $x,y \geq 0$ our system reads:

\begin{align}\label{system1}
\begin{cases}
\dot x &= x(x^a -y)\\
\dot y &= y \left ( \frac{1}{2}x^a -y \right )
\end{cases}
\end{align}

We remark that the right-hand side is at least $C^1$, but not $C^\infty$ if $a \not\in \N$.

It is easy to see that $0$ is a non-hyperbolic equilibrium of the system since the Jacobian is just the zero matrix. Both coordinate axes are invariant: for $y=0$ we have $\dot x = x^{a+1}>0$, so the $x$-axis belongs to the unstable set of $0$. Similarly, for $x=0$ we have $\dot y = -y^2<0$, so the $y$-axis belongs to the stable set of $0$.

The $x$- and $y$-nullclines off the coordinate axes are given by:

$$ \dot x = 0 \quad \Leftrightarrow \quad y= x^a \qquad \text{ and } \qquad \dot y = 0 \quad \Leftrightarrow \quad y=\frac{1}{2}x^a $$

This enables us to sketch the dynamics of system~(\ref{system1}) as in Figure~\ref{fig-greater1}. We now proceed to state and prove our result about the stability index.

\begin{figure}[!htb]
 \centering
\begin{tikzpicture}
\draw[thick] (-1,0) -- (7,0) node[anchor=north]{$x$};
\draw[thick] (0,-1) -- (0,5) node[anchor=east]{$y$};
\draw[dashed, thick] (0,0) parabola (6,2);
\draw[dashed, thick] (0,0) parabola (6,4);
\draw[dotted, thick] (0,0) rectangle (3,3);
\node at (0,0) [circle,fill,inner sep=1.5pt]{};
\draw [thick, -{Stealth[scale=1.5]}](0,4) -- (0,2);
\draw [thick, -{Stealth[scale=1.5]}](3,0) -- (4,0);
\draw [thick, -{Stealth[scale=1.5]}](2.6,0.5) -- (3.6,0.5);
\draw [thick, -{Stealth[scale=1.5]}](3.9,1) -- (4.9,1);
\draw [thick, -{Stealth[scale=1.5]}](4.8,1.5) -- (5.8,1.5);
\draw [thick, -{Stealth[scale=1.5]}](3.5,1.8) -- (3.5,0.8);
\draw [thick, -{Stealth[scale=1.5]}](4.5,2.7) -- (4.5,1.7);
\draw [thick, -{Stealth[scale=1.5]}](5.5,3.7) -- (5.5,2.7);
\node at (6,4.3) { $y=x^a$};
\node at (7,2.1) { $y=\frac{1}{2}x^a$};
\node at (2,3.3) { $B_\varepsilon(0)$};
\end{tikzpicture}
 \caption{Nullclines for system~(\ref{system1}) with $a>1$.}
 \label{fig-greater1}
\end{figure}
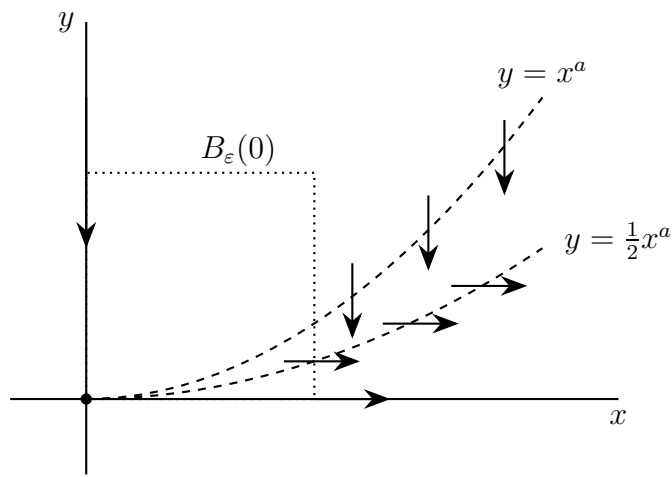

\begin{proposition}\label{prop1}
In system~(\ref{system1}), for $a >1$ the stability index of the origin is $\sigma(0)=a-1>0$.
\end{proposition}

\begin{proof}
From Figure~\ref{fig-greater1} it is clear that all points $(x,y)$ with $y<x^a$ do not belong to the basin of attraction $\B(0)$. This enables our first estimate:

$$\ell(B_\varepsilon(0) \cap \B(0)) \leq \varepsilon^2 - \int\limits_0^\varepsilon x^a \dx = \varepsilon^2 - \frac{1}{1+a} \varepsilon^{1+a}$$
and therefore 
$$\Sigma_\varepsilon(0) = \frac{\ell (B_\varepsilon(0) \cap \B(0)) }{\ell(B_\varepsilon(0))} \leq \frac{1}{\varepsilon^2} \left(\varepsilon^2 - \frac{1}{1+a}  \varepsilon^{1+a} \right) = 1- \frac{1}{1+a} \varepsilon^{a-1},$$
or equivalently
$$1-\Sigma_\varepsilon(0) \geq \frac{1}{1+a} \varepsilon^{a-1}.$$
Hence
$$\sigma_+(0)= \lim_{\varepsilon \to 0} \frac{\ln (1 - \Sigma_\varepsilon(0))}{\ln (\varepsilon)} \leq \lim_{\varepsilon \to 0} \frac{\ln (\varepsilon^{a-1})}{\ln (\varepsilon)} =a-1,$$
which finally implies $\sigma(0)=\sigma_+(0)-\sigma_-(0) \leq a-1$.

\medbreak

For the other inequality we show that there is a constant $k>1$ such that all $(x,y)$ with $y>kx^a$ belong to $\B(0)$, in fact, even to all $\B_\delta(0)$ with suitable $\delta >0$. In other words: we show that this region is forward invariant under the dynamics of system~(\ref{system1}) and all trajectories in it converge to the origin.

We claim that for a given $a>1$ a choice of $k > \frac{a-\frac{1}{2}}{a-1}>1$ suffices. This we prove by showing that the vector $(\dot x, \dot y)$ in this region always points downwards and ``to the left'' of the curve $(x,kx^a)$, which means the corresponding solution is for all positive times confined between $(x,kx^a)$ and the $y$-axis, and thus must limit to $0$. To see this, first note that clearly $\dot y <-\frac{1}{2}y^2 <0$ in this region. Furthermore, we calculate that the angle $\alpha$ between $(\dot x, \dot y)$ and the normal vector $(-akx^{a-1},1)$ is always in $(-\frac{\pi}{2}, \frac{\pi}{2})$ along $(x,kx^a)$, see Figure~\ref{fig-angle}. To that end, consider the scalar product:

\begin{figure}[!htb]
 \centering
\begin{tikzpicture}
\draw[thick] (-1,0) -- (7,0) node[anchor=north]{$x$};
\draw[thick] (0,-1) -- (0,5) node[anchor=east]{$y$};
\node at (0,0) [circle,fill,inner sep=1.5pt]{};
\draw[dashed, thick] (0,0) parabola (6,4.5);
\draw [thick](3,1) -- (5,3);
\draw [thick,dotted](4,2) -- (3,3);
\draw[thick,shift={(2.2 cm,-1.85 cm)},rotate=40] (3.8,2.4) arc (90:160:0.5cm);
\draw [thick, -{Stealth[scale=1.5]}](5,2.2) -- (2,1.6);
\draw [thick, -{Stealth[scale=1.5]}](0,4) -- (0,2);
\draw [thick, -{Stealth[scale=1.5]}](3,0) -- (4,0);
\node at (6.4,4.8) { $y=kx^a$};
\node at (3.1,2.2) { $\alpha$};
\end{tikzpicture}
\caption{The angle $\alpha$ between the (dotted) normal vector to $(x,kx^a)$ and the flow of system~(\ref{system1}).}
\label{fig-angle}
\end{figure}
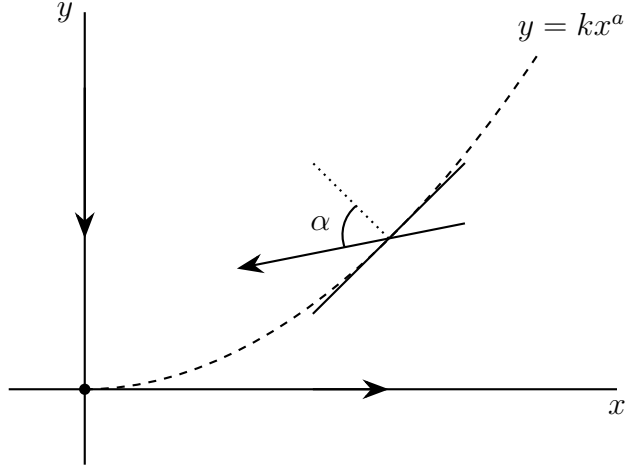

\begin{align*}
\langle (\dot x , \dot y),(-akx^{a-1},1) \rangle &= -akx^a(x^a-y) + y\left(\frac{1}{2}x^a-y \right)\\
&=-akx^a(x^a-kx^a) + kx^a \left(\frac{1}{2}x^a-kx^a \right)\\
&=kx^{2a} \left(a(k-1)+\frac{1}{2}-k \right)\\
&=kx^{2a} \left(k(a-1)-a+\frac{1}{2} \right),
\end{align*}
which is positive for all $x>0$ if and only if $k>1$ is chosen as above. Such a choice is obviously possible for any $a>1$. An analogous calculation to that at the beginning of this proof now yields $\sigma_+(0) \geq a-1$ and therefore $\sigma(0) \geq a-1$. Therefore, $\sigma(0)=a-1$ as claimed.
\end{proof}

\begin{corollary}
Given any $s>0$, set $a:=s+1>1$ to obtain $\sigma(0)=s$ in system~(\ref{system1}).
\end{corollary}

\begin{corollary}
For $a>1$ the origin in system~(\ref{system1}) is e.a.s.
\end{corollary}

\subsection{Negative Stability Indices}
We now strive for a similar result with negative stability indices. An analogous calculation for system~(\ref{system1}) with $a<1$ does not yield the desired flow, since no suitable $k$ can be found to obtain a positive scalar product as above: we would need $k>1$ as before, but with $a<1$ obtaining a positive scalar product requires $k<\frac{a-\frac{1}{2}}{a-1}<1$.

However, with $a\in(0,1)$ the following modification of system~(\ref{system1}) does the job:
\begin{align}\label{system2}
\begin{cases}
\dot x &= x(\frac{1}{2}x^a -y)\\
\dot y &= y^2 \left (x^a -y \right )
\end{cases}
\end{align}

Note that the smoothness of system~(\ref{system2})  is most severely limited by the $x$-term in the $y$-equation: since $a \in (0,1)$, the derivative of the second equation with respect to $x$ is undefined at the origin. It is also worth pointing out that a stronger contraction in the $y$-direction than in system~(\ref{system1})  is required to achieve the desired result, as becomes apparent in the calculations below.

As before the coordinate axes are invariant, and for $y=0$ we have $\dot x = \frac{1}{2}x^{a+1}>0$, so expanding dynamics on the $x$-axis; while for $x=0$ we have $\dot y=-y^3<0$, so contracting dynamics on the $y$-axis. Note that the position of the $x$- and $y$-nullclines has been reversed compared to system~(\ref{system1}), and we may sketch the phase portrait as in Figure~\ref{fig-less1}.

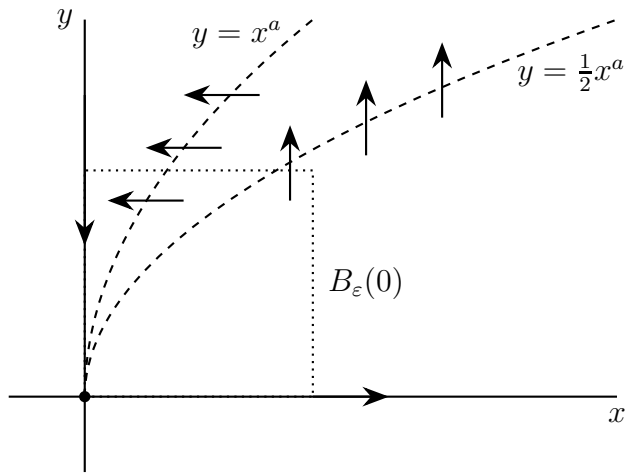
\begin{figure}[!htb]
 \centering
\begin{tikzpicture}
\draw[thick] (-1,0) -- (7,0) node[anchor=north]{$x$};
\draw[thick] (0,-1) -- (0,5) node[anchor=east]{$y$};
\draw[dashed, thick, rotate=90] (0,0) parabola (5,-3);
\draw[dashed, thick, rotate=90] (0,0) parabola (5,-7);
\draw[dotted, thick] (0,0) rectangle (3,3);
\node at (0,0) [circle,fill,inner sep=1.5pt]{};
\draw [thick, -{Stealth[scale=1.5]}](0,4) -- (0,2);
\draw [thick, -{Stealth[scale=1.5]}](3,0) -- (4,0);
\draw [thick, -{Stealth[scale=1.5]}](1.3,2.6) -- (0.3,2.6);
\draw [thick, -{Stealth[scale=1.5]}](1.8,3.3) -- (0.8,3.3);
\draw [thick, -{Stealth[scale=1.5]}](2.3,4) -- (1.3,4);
\draw [thick, -{Stealth[scale=1.5]}](2.7,2.6) -- (2.7,3.6);
\draw [thick, -{Stealth[scale=1.5]}](3.7,3.2) -- (3.7,4.2);
\draw [thick, -{Stealth[scale=1.5]}](4.7,3.7) -- (4.7,4.7);
\node at (2,4.8) { $y=x^a$};
\node at (6.4,4.3) { $y=\frac{1}{2}x^a$};
\node at (3.7,1.5) { $B_\varepsilon(0)$};
\end{tikzpicture}
\caption{Nullclines for system~(\ref{system2}) with $a<1$.}
\label{fig-less1}
\end{figure}

\begin{proposition}\label{prop2}
In system~(\ref{system2}), for $a < 1$ the stability index of the origin is $\sigma(0)=1-\frac{1}{a}<0$.
\end{proposition}

\begin{proof}
We argue in the same way as in the proof of Proposition~\ref{prop1} but with reversed justifications for the two inequalities: first observe from Figure~\ref{fig-less1} that for $y>x^a$ we have $\dot x, \dot y <0$ and therefore all such points belong to the (local) basin of attraction of the origin. Thus, we obtain:
$$\ell(B_\varepsilon(0) \cap \B(0)) \geq \int\limits_0^\varepsilon x^{\frac{1}{a}} \dx = \frac{a}{1+a} \varepsilon^{1+\frac{1}{a}}$$
and therefore 
$$\Sigma_\varepsilon(0) = \frac{\ell (B_\varepsilon(0) \cap \B(0)) }{\ell(B_\varepsilon(0))}  \geq  \frac{1}{\varepsilon^2} \frac{a}{1+a}  \varepsilon^{1+\frac{1}{a}} = \frac{a}{1+a}  \varepsilon^{\frac{1}{a}-1},$$
hence
$$\sigma_-(0)= \lim_{\varepsilon \to 0} \frac{\ln (\Sigma_\varepsilon(0))}{\ln (\varepsilon)} \leq \lim_{\varepsilon \to 0} \frac{\ln (\varepsilon^{\frac{1}{a}-1})}{\ln (\varepsilon)} =\frac{1}{a}-1,$$
which finally implies $\sigma(0)=\sigma_+(0)-\sigma_-(0) \geq \ 1-\frac{1}{a}$.

\medbreak

For the other inequality, we also proceed in a similar way as before, showing that along $(x,kx^a)$ the angle between $(\dot x, \dot y)$ and the normal vector $(akx^{a-1},-1)$ is in $(-\frac{\pi}{2}, \frac{\pi}{2})$ for suitable $0<k<\frac{1}{2}$. This implies that the region with $y<kx^a$ is forward invariant under the dynamics of system~(\ref{system2}). Moreover, solutions with initial conditions in it do not limit to the origin in forward time and thus do not belong to $\B(0)$, which enables our second estimate for the stability index. Again we consider the scalar product:

\begin{align*}
\langle (\dot x , \dot y),(akx^{a-1},-1) \rangle &= akx^a \left(\frac{1}{2}x^a-y \right) - y^2(x^a-y)\\
&=akx^a \left( \frac{1}{2}x^a-kx^a \right) - (kx^a)^2(x^a-kx^a)\\
&=kx^{2a} \left(a\left(\frac{1}{2}-k\right)-kx^a(1-k) \right).
\end{align*}

The second term in parentheses goes to zero when $x \to 0$, while the first one is constant in $x$ and positive for $0<k<\frac{1}{2}$. Thus, with $k \in \left(0, \frac{1}{2} \right)$ the scalar product is positive for sufficiently small $x>0$. Again, similar calculations as above now yield $\sigma_-(0) \geq \frac{1}{a}-1$ and thus finally $\sigma(0)=1-\frac{1}{a}<0$.  
\end{proof}

\begin{corollary}
Given any $s<0$, set $a:=\frac{1}{1-s}\in(0,1)$ to obtain $\sigma(0)=s$ in system~(\ref{system2}).
\end{corollary}

\begin{corollary}
For $a<1$ the origin in system~(\ref{system2}) is f.a.s., but not e.a.s.
\end{corollary}

With Propositions~\ref{prop1}~and~\ref{prop2} we have established that in these systems of equations we can obtain any positive or negative number as the stability index of the origin.

\subsection{Infinite Stability Indices}
More generally, instead of $x \mapsto x^a$ let us now take any function $x \mapsto \phi(x)$ and consider the following system for $x,y \geq 0$: 

\begin{align}\label{system3}
\begin{cases}
\dot x &= x \left (y- \frac{1}{2}\phi(x)\right ) \\
\dot y &= y (y - \phi(x))
\end{cases}
\end{align}

The smoothness of system~(\ref{system3}) is determined by the smoothness of $\phi$. If $\phi$ is non-negative and vanishes only at $0$, we can draw similar initial conclusions as above: the coordinate axes are invariant with contraction along the $x$-axis, where $\dot x=-\frac{1}{2}x\phi(x)>0$; and expansion along the $y$-axis, where $\dot y = y^2$.  Looking at the nullclines we obtain the sketch of the dynamics in Figure~\ref{fig-supcusp}.

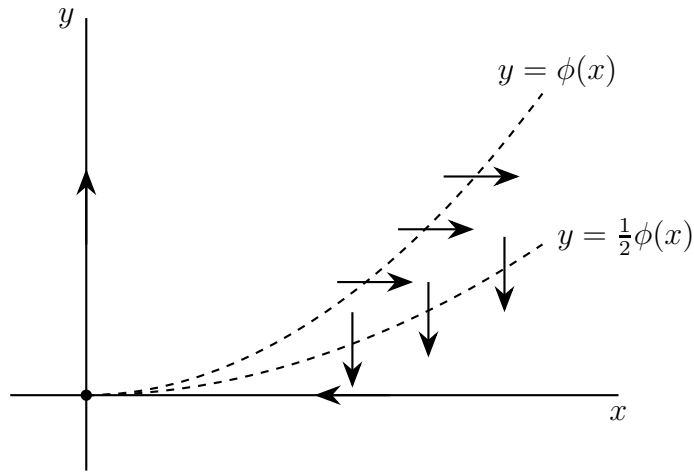
\begin{figure}[!htb]
 \centering
\begin{tikzpicture}
\draw[thick] (-1,0) -- (7,0) node[anchor=north]{$x$};
\draw[thick] (0,-1) -- (0,5) node[anchor=east]{$y$};
\draw[dashed, thick] (0,0) parabola (6,2);
\draw[dashed, thick] (0,0) parabola (6,4);
\node at (0,0) [circle,fill,inner sep=1.5pt]{};
\draw [thick, {Stealth[scale=1.5]}-](0,3) -- (0,2);
\draw [thick, {Stealth[scale=1.5]}-](3,0) -- (4,0);
\draw [thick, {Stealth[scale=1.5]}-](4.3,1.5) -- (3.3,1.5);
\draw [thick, {Stealth[scale=1.5]}-](5.1,2.2) -- (4.1,2.2);
\draw [thick, {Stealth[scale=1.5]}-](5.7,2.9) -- (4.7,2.9);
\draw [thick, {Stealth[scale=1.5]}-](3.5,0.1) -- (3.5,1.1);
\draw [thick, {Stealth[scale=1.5]}-](4.5,0.5) -- (4.5,1.5);
\draw [thick, {Stealth[scale=1.5]}-](5.5,1.1) -- (5.5,2.1);
\node at (6.2,4.3) { $y=\phi(x)$};
\node at (7.1,2.1) { $y=\frac{1}{2}\phi(x)$};
\end{tikzpicture}
 \caption{Dynamics for system~(\ref{system3}).}
 \label{fig-supcusp}
\end{figure}

We now pick a specific function for $\phi$ which is used in \cite{Lohse2015} to show that it is possible to have an equilibrium with a stability index equal to $+\infty$, but a local stability index equal to $-\infty$. This is achieved by making the equilibrium globally attracting, but confining the local basin of attraction within the region where $y<\phi(x)$. With this choice of $\phi$ for system~(\ref{system3}), we obtain the same extreme discrepancy between the local and non-local stability index of the origin, but achieve a higher degree of smoothness of the system than in \cite{Lohse2015}.

\begin{proposition}\label{prop-infty}
In system~(\ref{system3}) define $\phi$ as $\phi(x)=(2x+1)\exp(-\frac{1}{x})$ for $x>0$ and $\phi(0)=0$. Then we have $\sigma_\loc(0)=-\infty$ and $\sigma(0)=+\infty$.
\end{proposition}

\begin{proof}
We start with the claim about the local stability index. It is clear from Figure~\ref{fig-supcusp} that all $(x,y)$ with $y<\phi(x)$ belong to $\B(0)$, even to $\B_\delta(0)$ for suitable $\delta>0$. For $(x,y)$ with $y>\phi(x)$ we have $\dot x, \dot y >0$, so these trajectories first move away from the origin in both coordinates and do not belong to $\B_\delta(0)$ for sufficiently small $\delta>0$. Thus, by the same arguments as in \cite{Lohse2015}, we have $\sigma_\loc(0)=-\infty$.

To prove the second claim we show that in system~(\ref{system3}) all trajectories off the coordinate axes limit to $0$ in forward time and are thus homoclinic to the origin. In fact, because of the above it suffices to ensure that all trajectories starting with $y>\phi(x)$ eventually cross the graph of $\phi$.

To this end we show that $V(x,y):=\frac{x}{y}$ is a Lyapunov function for system~(\ref{system3}):
\begin{align*}
\frac{\partial}{\partial t}V(x,y)&=\frac{\dot x y - x \dot y}{y^2}\\
&=\frac{x(y-\frac{1}{2}\phi(x))y - xy(y-\phi(x))}{y^2}\\
&=\frac{x}{y} \left( y- \frac{1}{2}\phi(x) - (y-\phi(x)) \right)\\
&=\frac{x\phi(x)}{2y}\\
&>0
\end{align*}
Thus, $V$ increases along solutions to system~(\ref{system3}). The level sets of $V$ are straight lines through the origin, with the values of $V$ increasing as the slope of these lines decreases. Since the derivative above is bounded away from zero off the coordinate axes, solutions cross level sets of $V$ with non-vanishing speed and thus every solution eventually crosses the graph of $\phi$, therefore converging to the origin. Thus, $\sigma(0)=+\infty$ as claimed.
\end{proof}

Note that the corresponding example in \cite{Lohse2015} has a right-hand side that is only continuous, not differentiable. Our choice of $\phi$ in Proposition~\ref{prop-infty} makes system~(\ref{system3}) $C^\infty$, so we provide a smooth example of this kind.
\section{Concluding Remarks}\label{sec-comments}
We have discussed two families of systems of ordinary differential equations on $\R^2$ that possess a non-hyperbolic equilibrium with an arbitrary real number $s \in \R \setminus\{0\}$ as its stability index. While we give an explicit construction for any such $s$, it is worth pointing out that similar results can be obtained by taking system~(\ref{system1}) or (\ref{system2}) with a fixed parameter $a$ and transforming it through $(x,y)=(u^p,v)$. This coordinate change maps a curve given by $y=kx^a$ to that given by $v=kx^{pa}$ and thus yields a different stability index. For example, if $a>1$ is fixed and system~(\ref{system1}) is transformed with $p \in \R$ such that $pa>1$, then it follows directly from Proposition~\ref{prop1} that $\sigma(0)=pa-1$ in the transformed system.

Generalizing our construction and employing results from \cite{Lohse2015}, in Proposition~\ref{prop-infty} we have designed a system with a strongly attracting equilibrium ($\sigma(0)=+\infty$) that is far from being asymptotically stable ($\sigma_\loc(0)=-\infty$). In contrast to earlier such examples ours has a $C^\infty$ right-hand side, answering an open question posed in \cite{Lohse2015}.

In section~\ref{sec-example} we have not considered the case $\sigma(0)=0$. However, it is straightforward to write down such a system: one simply needs to make sure that $\Sigma_\varepsilon(0)$ is constant, i.e.\ independent of $\varepsilon>0$. This is the case if the basin of attraction is linearly bounded, see e.g.\ the piecewise linear vector field on $\R^2$ displayed in Figure~\ref{fig-zero}, where we have $\Sigma_\varepsilon(0)=\frac{1}{4}$ for all $\varepsilon>0$.

Our work establishes explicit examples for non-asymptotically stable equilibria that are fragmentarily or essentially asymptotically stable. This may prove useful in future endeavors to develop more complicated systems with heteroclinic connections that possess a prescribed level of stability, thus extending previous efforts towards the design of systems with a desired connection structure between equilibria, see e.g.\ \cite{AshwinCastroLohse2020, AP2013}.

\begin{figure}[!htb]
 \centering
\begin{tikzpicture}
\draw[thick] (-4,0) -- (4,0) node[anchor=north]{$x$};
\draw[thick] (0,-4) -- (0,4) node[anchor=east]{$y$};
\draw[dotted, thick] (-1.8,-1.8) rectangle (1.8,1.8);
\node at (0,0) [circle,fill,inner sep=1.5pt]{};
\draw [thick, -{Stealth[scale=1.5]}](0,2) -- (0,3);
\draw [thick, -{Stealth[scale=1.5]}](0,-3) -- (0,-2);
\draw [thick, -{Stealth[scale=1.5]}](2,0) -- (3,0);
\draw [thick, -{Stealth[scale=1.5]}](-3,0) -- (-2,0);
\draw [thick, -{Stealth[scale=1.5]}](0,0) -- (1,3);
\draw [thick, -{Stealth[scale=1.5]}](0,0) -- (2,2);
\draw [thick, -{Stealth[scale=1.5]}](0,0) -- (3,1);
\draw [thick](-3,-1) -- (0,0);
\draw [thick](-2,-2) -- (0,0);
\draw [thick](-1,-3) -- (0,0);
\draw [thick, -{Stealth[scale=1.5]}](-3,-1) -- (-2,-2/3);
\draw [thick, -{Stealth[scale=1.5]}](-2,-2) -- (-1,-1);
\draw [thick, -{Stealth[scale=1.5]}](-1,-3) -- (-2/3,-2);
\draw[thick, domain=-3.5:-0.4,smooth,variable=\t] plot ({\t},{-1.5/\t});
\draw[thick, domain=-3.5:-0.15,smooth,variable=\t] plot ({\t},{-0.5/\t});
\draw [thick, -{Stealth[scale=1.5]}](-0.72,0.7) -- (-0.62,0.8);
\draw [thick, -{Stealth[scale=1.5]}](-1.22,1.23) -- (-1.11,1.33);
\draw[thick, domain=0.4:3.5,smooth,variable=\t] plot ({\t},{-1.5/\t});
\draw[thick, domain=0.15:3.5,smooth,variable=\t] plot ({\t},{-0.5/\t});
\draw [thick, -{Stealth[scale=1.5]}] (0.67,-0.75) -- (0.77,-0.64);
\draw [thick, -{Stealth[scale=1.5]}] (1.18,-1.27) -- (1.28,-1.16);
\node at (3,3) { $\dot x = x, \ \dot y = y$};
\node at (-3,3) { $\dot x = -x, \ \dot y = y$};
\node at (3,-3) { $\dot x = x, \ \dot y = -y$};
\node at (-3,-3) { $\dot x = -x, \ \dot y = -y$};
\node at (-1.6,2.1) { $B_\varepsilon(0)$};
\end{tikzpicture}
 \caption{A system with $\sigma(0)=0$.}
 \label{fig-zero}
\end{figure}
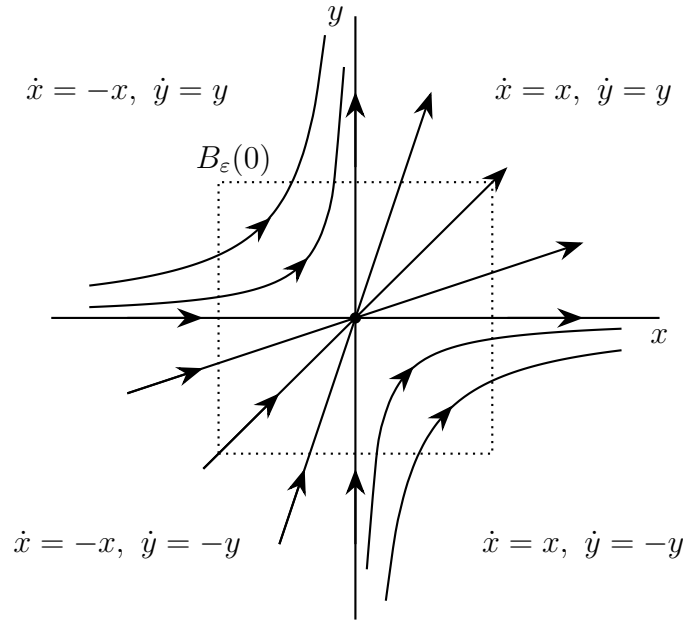

\paragraph{Acknowledgements:}
The author is grateful for insightful comments by two anonymous reviewers of an earlier version of this work.

\FloatBarrier

\end{document}